\documentclass[12pt,a4paper]{amsart}
\usepackage{amsfonts,amscd,amssymb,amsthm}
\usepackage{amsfonts}
\usepackage{amssymb}
\usepackage{amsmath}
\usepackage{amsxtra}
\usepackage{amscd}
\usepackage{amsthm}
\usepackage{eucal}
\usepackage{array}
\usepackage{graphicx}
\usepackage{tikz}
\usepackage{mathrsfs} 
\usepackage{enumerate}
\usepackage{hyperref}
\usepackage{bm} 

\newtheorem{theorem}{Theorem}[section]
\newtheorem{lemma}[theorem]{Lemma}
\newtheorem{proposition}[theorem]{Proposition}
\newtheorem{corollary}[theorem]{Corollary}

\theoremstyle{remark}
\newtheorem{remark}[theorem]{Remark}

\theoremstyle{definition}

\DeclareMathOperator{\reg}{reg}
\DeclareMathOperator{\coc}{cochord}
\DeclareMathOperator{\mat}{mat}
\DeclareMathOperator{\pd}{pd}
\DeclareMathOperator{\indm}{indm}
\DeclareMathOperator{\bpt}{BPT}
\DeclareMathOperator{\tree}{TREE}

\begin{document}

\title{The size of Betti tables of edge ideals arising from bipartite graphs}

\author[N. Erey]{Nursel Erey}
\author[T. Hibi]{Takayuki Hibi}
\address{Nursel Erey, Gebze Technical University, Department of Mathematics, 41400 Kocaeli, Turkey}
\email{nurselerey@gtu.edu.tr}

\address{Takayuki Hibi, Department of Pure and Applied Mathematics, Graduate School of Information Science and Technology, Osaka University, Suita, Osaka 565--0871, Japan}
\email{hibi@math.sci.osaka-u.ac.jp}



\begin{abstract}
 Let $\pd(I(G))$ and $\reg(I(G))$ respectively denote the projective dimension and the regularity of the edge ideal $I(G)$ of a graph $G$. For any positive integer $n$, we determine all pairs $(\pd(I(G)),\, \reg(I(G)))$ as $G$ ranges over all connected bipartite graphs on $n$ vertices. 
\end{abstract}

\maketitle

\section{Introduction}
Let $G$ be a finite simple graph with the vertex set $V(G)=\{x_1,\dots ,x_n\}$. Let $S=\Bbbk[x_1,\dots ,x_n]$ be the polynomial ring in $n$ variables over a field $\Bbbk$. The \textit{edge ideal} of $G$, denoted by $I(G)$, is the monomial ideal generated by the monomials $x_ix_j$ such that $\{x_i,x_j\}$ is an edge of $G$. Edge ideals of bipartite graphs were studied in the literature for several purposes. Fernández-Ramos and Gimenez \cite{FG} gave a characterization of bipartite graphs whose edge ideal has regularity $3$. When $G$ is an unmixed bipartite graph, Kummini \cite{Ku} described the regularity of $I(G)$ in terms of the induced matching number of $G$ and Kimura \cite{K2} gave a combinatorial description of the projective dimension of $I(G)$ via complete bipartite subgraphs satisfying certain conditions. Van Tuyl \cite{VT} provided a formula for the regularity of the edge ideal of a sequentially Cohen-Macaulay bipartite graph in terms of the induced matching number of the graph. Jayanthan et al. \cite{JNS} computed the regularity of powers of edge ideals for several subclasses of bipartite graphs. Herzog and Hibi \cite{HH_bipartite} classified all bipartite graphs which are Cohen-Macaulay and Van Tuyl and Villarreal \cite{VTV} classified those which are shellable.

In a recent article, Hà and Hibi \cite{HaHibi} considered the following problem: Given a graph $G$ on $n$ vertices, what are the possible values of $(\pd(S/I(G)),\, \reg(S/I(G)))$? They determined all such pairs when $\pd(S/I(G))$ attains its minimum possible value $2\sqrt{n}-2$ or when $\reg(S/I(G))$ attains its minimum possible value $1$. Hibi et al. \cite{HKKMVT} determined all tuples consisting of the values of depth, regularity, dimension and the degree of the $h$-polynomial of $S/I(G)$ as $G$ ranges over all Cameron-Walker graphs on $n$ vertices. Similar type of problems were studied recently in \cite{FKVT, HKM, HKMVT, HMVT}. 

In this article, we determine all pairs $(\pd(I(G)),\, \reg(I(G)))$ as $G$ ranges over all connected bipartite graphs on $n$ vertices. To state our main result precisely, for any positive integer $n$ we denote by $\bpt(n)$ the set of connected bipartite graphs on the vertices $\{x_1,\dots ,x_n\}$. We define
	\[\displaystyle \bpt_{\pd}^{\reg}(n)=\{(\pd(S/I(G)),\reg(S/I(G))): G \in \bpt(n)\}
\]
which is the set of sizes of Betti tables of $S/I(G)$ as $G$ ranges over all connected bipartite graphs on $n$ vertices. Our main result is then the following theorem:
\begin{theorem}[Theorem~\ref{thm:main theorem}]
	Let $n\geq 4$ be an integer. Then
	\[\displaystyle \bpt_{\pd}^{\reg}(n)=\{(p,r)\in \mathbb{Z}^2 : 1\leq r <  \Big\lfloor \frac{n}{2}\Big\rfloor, \, \Big\lceil \frac{n}{2} \Big\rceil \leq p \leq n-2 \}\cup \{(n-1,1)\} \cup A_n \]
	where $A_n=\emptyset$ if $n$ is even and, $A_n=\{(\lceil n/2 \rceil, \lfloor n/2 \rfloor)\}$ if $n$ is odd.
\end{theorem}

We make use of the graph parameters (induced) matching number, co-chordal cover number and maximum size of minimal vertex covers to bound the regularity and projective dimension. Along the way, we describe all the pairs $(\pd(I(G)),\, \reg(I(G)))$ as $G$ ranges over all trees on $n$ vertices.
	
\section{Preliminaries}
\subsection{Graph theory background}
Given a finite simple graph $G$, we denote by $V(G)$ and $E(G)$ respectively the vertex set and the edge set of $G$. We say a vertex $u$ is \textit{neighbor} of (or, \textit{adjacent} to) another vertex $v$ if $\{u,v\}\in E(G)$. We denote by $N(u)$ the set consisting of all neighbors of $u$ in $G$. We define $N[u]$ by $N[u]=N(u)\cup \{u\}$. We call a vertex \textit{isolated} if it has no neighbors.

A graph $H$ is called a \textit{subgraph} of $G$ if $V(H)\subseteq V(G)$ and $E(H)\subseteq E(G)$. A subgraph $H$ of $G$ is called an \textit{induced subgraph}  if for any two vertices $u,v$ in $H$, $\{u,v\}\in E(H)$ if $\{u,v\}\in E(G)$. If $U$ is a subset of $V(G)$, we define the induced subgraph of $G$ on $U$ as the subgraph whose vertex set is $U$ and whose edge set is  $\{\{x,y\} : x,y\in U \text{ and } \{x,y\} \in E(G)\}$. Moreover, for any $W\subseteq V(G)$, we denote by $G-W$ the induced subgraph of $G$ on $V(G)\setminus W$. To simplify the notation, if $W=\{x\}$ consists of a single vertex, then we write $G-x$ for $G-W$. The \textit{complement} of $G$, denoted by $G^c$, is a graph that has the same vertices as $G$ such that $\{x,y\}\in E(G^c)$ if and only if $\{x,y\}\notin E(G)$.

A graph $G$ is called \textit{connected} if for every pair of vertices $x$ and $y$, there is a path in $G$ that starts at $x$ and ends at $y$. A maximal connected subgraph of $G$ is called a \textit{connected component} of $G$. We say $G$ is a \textit{forest} if $G$ has no cycle subgraphs. A connected forest is called a \textit{tree}. It is well-known that every tree on $n$ vertices has exactly $n-1$ edges. An \textit{independent set} in $G$ is a subset of vertices which contain no edges of $G$. A \textit{bipartite} graph is a graph that contains no odd cycles. The vertex set of a bipartite graph can be partitioned to two independent sets. A bipartite graph $G$ with vertex bipartition $V(G)=A\cup B$ is called a \textit{complete bipartite graph} if every vertex in $A$ is adjacent to every vertex in $B$. A graph is called \textit{chordal} if it has no induced cycles of length greater than three. A graph $G$ is called \textit{co-chordal} if $G^c$ is chordal.

A \textit{matching} of $G$ is a collection of edges which are pairwise disjoint. The \textit{matching number} of $G$, denoted by $\mat(G)$, is defined by
\[\mat(G)=\max\{|M|: M \text{ is a matching of } G\}.\]
A matching $M=\{e_1,\dots ,e_k\}$ of $G$ is called an \textit{induced matching} of $G$ if the induced subgraph of $G$ on $\cup_{i=1}^ke_i$ has exactly $k$ edges.  The \textit{induced matching number} of $G$, denoted by $\indm(G)$, is the maximum cardinality of an induced matching of $G$. A \textit{perfect matching} of $G$ is a matching $M$ such that each vertex of $G$ belongs to some edge in $M$. The \textit{co-chordal cover number} of $G$, denoted by $\coc(G)$, is the minimum number of co-chordal subgraphs required to cover the edges of $G$, i.e.,
\[\displaystyle \coc(G)=\min\{r: E(G)=\bigcup_{i=1}^rE(H_i), \text{ each }H_i \text{ is a co-chordal subgraph of } G\}.\]

For any positive integer $n$, we denote $\{1,\dots ,n\}$ by $[n]$.

A \textit{vertex cover} $C$ of a graph $G$ is a subset of vertices such that every edge of $G$ contains a vertex from $C$. A vertex cover is called \textit{minimal} if no proper subset of it is a vertex cover. The maximum cardinality of a minimal vertex cover of $G$ is denoted by $\tau_{\max}(G)$.
\subsection{Algebra background}
Let $G$ be a graph with the vertex set $V(G)=\{x_1,\dots ,x_n\}$. Let $\Bbbk$ be a field and let $S=\Bbbk[x_1,\dots ,x_n]$ be the polynomial ring in $n$ variables over $\Bbbk$. The \textit{edge ideal} of $G$, denoted by $I(G)$, is the monomial ideal defined by
\[I(G)=(x_ix_j : \{x_i,x_j\} \text{ is an edge of } G).\]

Let $M$ be a finitely generated graded $S$-module. Then $M$ has a minimal graded free resolution of the form 

\begin{equation*}\label{eq:resolution}
0 \longrightarrow \bigoplus_{j \in \mathbb{Z}} S(-j)^{b_{p,j}(M)} \longrightarrow \cdots\longrightarrow 
\bigoplus_{j \in \mathbb{Z}}S(-j)^{b_{0,j}(M)} \longrightarrow M \longrightarrow 0 .
\end{equation*}
The numbers $b_{i,j}(M)$ are called the \textit{graded Betti numbers} of $M$. The \textit{projective dimension} of $M$, denoted by $\pd(M)$, is defined by
\[\pd(M)=\max\{i: b_{i,j}(M)\neq 0 \text{ for some } j\}.\]
The \textit{(Castelnuovo-Mumford) regularity} of $M$, denoted by $\reg(M)$, is defined by
\[\reg(M)=\max\{j-i: b_{i,j}(M)\neq 0 \}.\]

	\begin{theorem}\cite[Corollary~3.8]{BBHsurvey}\label{thm:regularity formula connected components}
		If $G$ is a graph with connected components $G_1,\dots ,G_r$, then
		\[\reg(S/I(G))=\sum_{i=1}^{r}\reg(S/I(G_i)).\]
	\end{theorem}

The following bounds on the regularity and projective dimension of edge ideals are well-known, see for example \cite[Lemma~3.1]{DHS} and \cite[Lemma~3.2]{DS}.
\begin{lemma}\cite{DHS, DS}\label{lem: key lemma}
	For any vertex $x$ of a graph $G$, the short exact sequence 
	\[0 \rightarrow \frac{S}{I(G):(x)}(-1) \rightarrow \frac{S}{I(G)} \rightarrow \frac{S}{I(G)+(x)} \rightarrow 0\]
	gives the following bounds for the regularity and projective dimension:
	\begin{enumerate}
		\item $\reg(S/I(G))\leq \max\{\reg(S/I(G-x)), \, \reg(S/I(G-N[x]))+1\}$,
		\item $\pd(S/I(G))\leq \max\{\pd(S/I(G-x))+1, \, \pd(S/I(G-N[x]))+|N(x)|\}$.
	\end{enumerate}
\end{lemma}
The \textit{Stanley-Reisner ideal} of a simplicial complex $\Delta$ is the squarefree monomial ideal generated by the monomials corresponding to non-faces of $\Delta$. The following theorem of Hochster \cite{Hoc} provides a formula for the graded Betti numbers of Stanley-Reisner ideals.
\begin{theorem}[Hochster’s Formula]\cite{Hoc}\label{thm:hoch} 
	Let $I_{\Delta}$ be the Stanley-Reisner ideal of a simplicial complex $\Delta$. If $i\geq 0$ and $u$ is a squarefree monomial, then
	\[b_{i,u}(I_{\Delta})=\dim_\Bbbk \tilde{H}_{\deg u-i-2}(\Delta[u]; \Bbbk) \]
	where $\Delta[u]=\{\sigma\in \Delta: \sigma \subseteq U\}$ and $U$ consists of those vertices that correspond to the variables dividing $u$.
\end{theorem}
If $G$ is a graph, the \textit{independence complex} of $G$ is a simplicial complex whose faces are independent sets of $G$. The edge ideal $I(G)$ of $G$ is the Stanley-Reisner ideal of the independence complex of $G$.

By a theorem of Terai \cite{T}, $\pd(S/I(G))$ is equal to the regularity of the Alexander dual of $I(G)$. Therefore, the projective dimension problem for edge ideals is equivalent to the regularity problem for so-called cover ideals. The next theorem  can be deduced from \cite[Corollary~3.3]{MV} or \cite[Corollary~8.2.14]{HH} both of which are stated in the more general setting of monomial ideals but in dual terms.
\begin{theorem}\cite{HH, MV}\label{thm: pd lower bound for any graph}
	For any graph $G$, $\pd(S/I(G))\geq \tau_{\max}(G)$. Moreover, the equality holds when $S/I(G)$ is sequentially Cohen-Macaulay. 
\end{theorem}
Since forests are known to be sequentially Cohen-Macaulay (see \cite{F} or \cite{FVT}), we have an exact formula for $\pd(S/I(G))$ when $G$ is a forest. This formula was proved independently by several authors in the literature.

\begin{theorem}\cite{K,Z}\label{thm: forest pd}
	If $G$ is a forest, then $\pd(S/I(G))=\tau_{\max}(G)$.
\end{theorem}

The following lower bound was also proved several times in the literature:

\begin{theorem}\cite{HVT, Kat, K, Z} \label{thm: reg lower bound for any graph}
	For any graph $G$, $\reg(S/I(G))\geq \indm(G)$.
\end{theorem}

When $G$ is a forest, the regularity can be described by $\indm(G)$. 
\begin{theorem}\cite{HVT, K, Z}\label{thm: forest reg}
	If $G$ is a forest, then $\reg(S/I(G))=\indm(G)$.
\end{theorem}

 Hà and Van Tuyl \cite{HVT} gave an upper bound for the regularity of edge ideal of any graph via the matching number of the graph:

\begin{theorem}\cite{HVT}\label{thm: matching number upper bound}
	For any graph $G$, $\reg(S/I(G))\leq \mat(G)$.
\end{theorem}

Woodroofe \cite{W} improved the upper bound in the previous theorem by replacing the matching number with the co-chordal cover number:
\begin{theorem}\cite{W}\label{thm: co-chord number upper bound}
	For any graph $G$, $\reg(S/I(G))\leq \coc(G)$.
\end{theorem}

\section{Edge Ideals of Bipartite Graphs}	
In the following lemma, we provide a rough estimate for the possible values of regularity and projective dimension of edge ideals of bipartite graphs.	
	
	\begin{lemma}\label{lem: coarse bound}
		Let $G$ be bipartite graph on $n\geq 2$ vertices which has no isolated vertices. Then
		\begin{enumerate}
			\item $\lceil n/2 \rceil \leq \pd(S/I(G)) \leq n-1$,
			\item $1\leq \reg(S/I(G)) \leq \lfloor n/2 \rfloor$.
		\end{enumerate}
	\end{lemma}
\begin{proof}
	Since $I(G)$ is generated in degree two, it is clear that $\pd(S/I(G))\leq n-1$ and $\reg(S/I(G))\geq 1$. Let $V(G)=A\cup B$ be a bipartition of the vertex set of $G$. Then either $A$ or $B$ has cardinality at least $\lceil n/2 \rceil$. Since $G$ has no isolated vertices, both $A$ and $B$ are minimal vertex covers. Hence $\tau_{\max}(G)\geq \lceil n/2 \rceil$ and $\pd(S/I(G))\geq \lceil n/2 \rceil$ follows from Theorem~\ref{thm: pd lower bound for any graph}. Lastly, $\reg(S/I(G)) \leq \lfloor n/2 \rfloor$ follows from Theorem~\ref{thm: matching number upper bound} as $\mat(G)$ cannot exceed the minimum of the cardinalities of $A$ and $B$.
\end{proof}
	
We will determine when the regularity upper bound in Lemma~\ref{lem: coarse bound} can be realized by connected bipartite graphs. It turns out that when $n$ is an even integer greater than two, no connected graph attains the regularity value in the upper bound. On the other hand, when $n$ is odd, we will see that the regularity upper bound is sharp and, in such case, the projective dimension is uniquely determined.

	\begin{theorem}\label{thm:even case}
		Let $G$ be a connected graph on $n\geq 4$ vertices where $n$ is even. If the matching number of $G$ is $n/2$, then $\coc(G)< n/2$.
	\end{theorem}
	\begin{proof}
		Let $\{e_1,\dots ,e_{n/2}\}$ be a matching of $G$. Then it is a perfect matching. Since $
		G$ is connected, we may assume that there is an edge $e$ such that $e\cap e_1 \neq\emptyset$ and $e\cap e_2\neq \emptyset$. Let $H_1$ denote the induced subgraph of $G$ on $e_1\cup e_2$. Furthermore, for each $2\leq i \leq n/2-1$ let $H_i$ be the subgraph of $G$ which consists of those edges $f\in E(G)$ with $f\cap e_{i+1}\neq \emptyset$. Then each $H_i$ is co-chordal and $E(G)=E(H_1)\cup \dots \cup E(H_{n/2-1})$. Hence $\coc(G)\leq n/2-1$.
	\end{proof}
	
	\begin{corollary}\label{cor: even regularity}
		If $G$ is a connected graph on $n\geq 4$ vertices where $n$ is even, then $\reg(S/I(G))<n/2$.
	\end{corollary}
	\begin{proof}
		By Theorem~\ref{thm: matching number upper bound} we have $\reg(S/I(G))\leq \mat(G)\leq n/2$. Assume for a contradiction $\reg(S/I(G))= n/2$. Then $\mat(G)=n/2$. From Theorem~\ref{thm: co-chord number upper bound} and Theorem~\ref{thm:even case} it follows that $\reg(S/I(G))\leq \coc(G)<n/2$, contradiction.
	\end{proof}
We can actually classify all graphs $G$ on even number of vertices for which $\reg(S/I(G))$ is equal to half the number of vertices: 
	\begin{corollary}\label{cor: even disconnected regularity}
		Let $G$ be a graph on $n$ vertices where $n$ is even. Then $\reg(S/I(G))=n/2$ if and only if $G$ consists of $n/2$ disjoint edges.
	\end{corollary}
	\begin{proof}
		Let us assume that $n\geq 4$ as the statement is clear otherwise. If $G$ consists of $n/2$ disjoint edges, then $\reg(S/I(G))=n/2$ follows from Theorem~\ref{thm:regularity formula connected components}. To show the converse, let $\reg(S/I(G))=n/2$. Then by Theorem~\ref{thm: matching number upper bound} the matching number of $G$ is $n/2$ and $G$ has a perfect matching. By Corollary~\ref{cor: even regularity} the graph $G$ is disconnected. Then every connected component of $G$ has a perfect matching. Let $G_1,\dots,G_r$ with $r\geq 2$ be the connected components of $G$. Let $|V(G_i)|=2k_i$ for each $i\in [r]$. Assume for a contradiction one of the connected components has at least two edges. We may assume that $k_1\geq 2$. Then by Corollary \ref{cor: even regularity} we have $\reg(S/I(G_1))<k_1$. Moreover, $\reg(S/I(G_i))\leq \mat(G_i)=k_i$ for each $i=2,\dots ,r$. Using Theorem~\ref{thm:regularity formula connected components} we get
		\[\reg(S/I(G))=\sum_{i=1}^{r}\reg(S/I(G_i)) < \sum_{i=1}^{r}k_i =n/2\]
		which is a contradiction.
	\end{proof}

We can now investigate the regularity upper bound in Lemma~\ref{lem: coarse bound} when $n$ is an odd integer. The following theorem classifies all bipartite graphs $G$ on $n$ vertices with $\reg(S/I(G))=(n-1)/2$.
 
	\begin{theorem}\label{thm: main theorem odd case}
		Let $G$ be a bipartite graph on $n$ vertices where $n$ is an odd number. Then $\reg(S/I(G))=\lfloor n/2 \rfloor$ if and only if $\indm(G)=\lfloor n/2 \rfloor$.
	\end{theorem}
	\begin{proof}
		Let $n=2k+1$. If $\indm(G)=k$, then $G$ is a forest and $\reg(S/I(G))=k$ by Theorem~\ref{thm: forest reg}.
		Now, suppose that $\reg(S/I(G))=k$. Since $\mat(G)\geq \reg(S/I(G))$ by Theorem~\ref{thm: matching number upper bound}, there exists a matching $M=\{e_1,\dots ,e_k\}$ of $G$. Let $x$ be the vertex of $G$ which does not belong to any edge in $M$. We may assume that $k\geq 2$ because $\indm(G)=1$ is clear when $k=1$. We consider two cases:
		
	\textit{Case 1:} Suppose that $x$ is an isolated vertex of $G$. Then \[k=\reg(S/I(G))=\reg(S/I(G-x))\] and by Corollary~\ref{cor: even disconnected regularity} it follows that $M$ is an induced matching of $G-x$. Hence $M$ is an induced matching of $G$.
		
	\textit{Case 2:} Suppose that $\{x,y\}$ is an edge of $G$ where $e_k=\{y,z\}$. By Lemma~\ref{lem: key lemma}
		\[\reg(S/I(G))\leq \max\{\reg(S/I(G-y)),\, \reg(S/I(G-N[y]))+1\}.\]
		Hence either $\reg(S/I(G-y))\geq k$ or $\reg(S/I(G-N[y]))\geq k-1$. Observe that $G-y$ is a bipartite graph such that one side of the bipartition has $k-1$ vertices. This implies that the matching number of $G-y$ is at most $k-1$. Therefore, $\reg(S/I(G-y))< k$ follows from Theorem~\ref{thm: matching number upper bound}. Thus we must have $\reg(S/I(G-N[y]))\geq k-1$. Similarly, since $\mat(G-N[y])\leq k-1$, it follows from Theorem~\ref{thm: matching number upper bound} that \[\reg(S/I(G-N[y]))= k-1=\mat(G-N[y]).\] Hence $N(y)=\{x,z\}$. Corollary~\ref{cor: even disconnected regularity} implies that $\{e_1,\dots ,e_{k-1}\}$ is an induced matching of $G-N[y]$ and thus it is induced matching of $G$. If $y$ is the only neighbor of $x$, then $\{e_1,\dots, e_{k-1}, \{x,y\}\}$ is an induced matching of $G$ and nothing is left to show. Suppose that $x$ has at least two neighbors, say $y$ and $u$. Without loss of generality, we may assume that $e_{k-1}=\{u,v\}$. By Lemma~\ref{lem: key lemma}
		\[\reg(S/I(G))\leq \max\{\reg(S/I(G-x)),\, \reg(S/I(G-N[x]))+1\}.\]
		Hence either $\reg(S/I(G-x))\geq k$ or $\reg(S/I(G-N[x]))\geq k-1$. Observe that $G-N[x]$ is a bipartite graph such that one side of the bipartition has at most $k-2$ vertices. This implies that the matching number of $G-N[x]$ is at most $k-2$. Therefore, $\reg(S/I(G-N[x]))< k-1$ by Theorem~\ref{thm: matching number upper bound}. Thus we must have $\reg(S/I(G-x))\geq k$. In fact, $\reg(S/I(G-x))=k$ because the matching number of $G-x$ is equal to $k$. Then Corollary~\ref{cor: even disconnected regularity} implies that $G-x$ consists of $k$ disjoint edges. Then $M$ is an induced matching of $G-x$. Thus $M$ is an induced matching of $G$.
	\end{proof}
	
	\begin{remark}
		In Theorem \ref{thm: main theorem odd case} the bipartite assumption cannot be dropped. Indeed, if $G$ is a cycle graph of length $5$, then $\reg(S/I(G))=2$ but $\indm(G)=1$. 
	\end{remark}
	
	\begin{corollary}\label{cor: odd regularity}
		Let $G$ be a connected bipartite graph on $n=2k+1$ vertices such that $k\geq 1$. Suppose that $\reg(S/I(G))=k$. Then $\pd(S/I(G))=k+1$.
	\end{corollary}
	\begin{proof}
		By Theorem \ref{thm: main theorem odd case} the induced matching number of $G$ is $k$. Let $M=\{e_1,\dots ,e_k\}$ be an induced matching of $G$. Let $x$ be the vertex of $G$ that does not belong to any edge in $M$. Then since $G$  is connected, for every $i\in [k]$, the vertex $x$ is adjacent to exactly one endpoint of $e_i$. Hence $G$ is a tree with $\tau_{\max}(G)=k+1$ and the proof is complete because of Theorem~\ref{thm: forest pd}.
	\end{proof}
The next result describes all connected bipartite graphs for which the projective dimension upper bound in Lemma~\ref{lem: coarse bound} can be realized.
\begin{proposition}\label{prop: pd max value}
	Let $G$ be a connected bipartite graph on $n\geq 2$ vertices. Then $\pd(S/I(G))=n-1$ if and only if $G$ is a complete bipartite graph. Moreover, in such case, $\reg(S/I(G))=1$. 
\end{proposition}
\begin{proof}
	If $G$ is a complete bipartite graph, then $\pd(S/I(G))=n-1$ and $\reg(S/I(G))=1$ was proved in \cite{J}. Now, suppose that $\pd(S/I(G))=n-1$. Since $G$ is connected, $I(G)\neq (0)$. Then $n-2=\pd(S/I(G))-1=\pd(I(G))$. Then there exists a squarefree monomial $u$ such that $b_{n-2,u}(I(G))\neq 0$. Since $I(G)$ is generated in degree two, the degree of $u$ must be $n$. Theorem~\ref{thm:hoch} implies $\dim_\Bbbk\tilde{H}_{0}(\Delta; \Bbbk)\neq 0$ where $\Delta$ is the independence complex of $G$. Then $\Delta$ is disconnected. Let $V(G)=A\cup B$ be a bipartition of the vertex set of $G$. Since $G$ has no isolated vertices, both $A$ and $B$ are facets of $\Delta$. Assume for a contradiction $G$ is not complete bipartite. Then there exists $a\in A$ and $b\in B$ such that $\{a,b\}\notin E(G)$. We now show that $\Delta$ is connected. First, observe that $A,\{a,b\},B$ is a chain from $A$ to $B$. Let $\sigma$ and $\tau$ be two facets of $\Delta$. Since both $A$ and $B$ are facets, we may assume that $\sigma \cap A\neq \emptyset$ and $\tau \cap A\neq \emptyset$. Then $\sigma, A, \tau$ is a chain from $\sigma$ to $\tau$ which shows that $\Delta$ is connected, a contradiction.
\end{proof}
	Now that we have determined when the upper bounds in Lemma~\ref{lem: coarse bound} can be attained by connected bipartite graphs, our next goal is to show that for any integers $p$ and $q$ with $\lceil n/2 \rceil \leq p <n-1$ and $1\leq r <\lfloor n/2 \rfloor$ there exists a connected bipartite graph with $\pd(S/I(G))=p$ and $\reg(S/I(G))=r$. We will show the existence of such graph in two steps (Theorem~\ref{thm: bipartite construction} and Theorem~\ref{thm: trees (r,p) description}) as our construction depends on whether $p+r$ exceeds $n$ or not.
	\begin{theorem}\label{thm: bipartite construction}
		Let $n, p$ and $r$ be integers with $3\leq r \leq n/2-1$ and $n-r<p<n-1$. Then there exists a connected bipartite graph $G$ on $n$ vertices such that $\reg(S/I(G))=r$ and $\pd(S/I(G))=p$.
	\end{theorem}
	\begin{proof}
		Since $r\leq n/2-1$ there exists an integer $t\geq 0$ such that $n=2r+2+t$. Moreover, $n-r<p$ implies that $r+t+3\leq p$. Thus, we may assume that $p=a+r+t$ for some $a\geq 3$. Note that $a\leq r$ because $p<n-1$. Let $G$ be the graph with the vertex set \[V(G)=\{u_1,\dots,u_r\}\cup \{v_1,\dots,v_r\}\cup \{x,y,z_1,\dots ,z_t\}  \]
		and the edge set
		\[E(G)=\{\{u_i,v_i\}: i\in [r]\} \cup \{\{x,v_i\}: i\in[a]\}\cup \{\{y,u_i\}: i\in [r]\} \cup\{\{y,z_i\}: i\in[t]\}.\]
		One can easily see that $G$ is a connected bipartite graph on $n$ vertices. We will now show that $G$ possesses the properties stated in the theorem. Observe that $\{\{u_i,v_i\}: i\in [r]\}$ is an induced matching of $G$. Therefore $\reg(S/I(G))\geq \indm(G)\geq r$ by Theorem~\ref{thm: reg lower bound for any graph}. On the other hand, by Lemma~\ref{lem: key lemma} we have
		\[\reg(S/I(G))\leq \max\{\reg(S/I(G-x)),\, \reg(S/I(G-N[x]))+1\}.\]
		Observe that $G-x$ is a tree with  $\indm(G-x)=r$ and $G-N[x]$ is a tree with $\indm(G-N[x])=\max\{1,r-a\}$. Therefore, by Theorem~\ref{thm: forest reg} we get
		\[\reg(S/I(G))\leq \max\{r,\, \max\{1, r-a\}+1\}\leq r \]
		as desired. To evaluate the projective dimension, first observe that by Theorem~\ref{thm: pd lower bound for any graph} \[\pd(S/I(G))\geq \tau_{\max}(G)\geq a+r+t.\] 
		On the other hand, by Lemma~\ref{lem: key lemma} we have
		\[\pd(S/I(G))\leq \max\{\pd(S/I(G-N[y]))+|N(y)|,\, \pd(S/I(G-y))+1\}.\]
		Moreover, $G-y$ is a forest with $\tau_{\max}(G-y)=r+1$ and $G-N[y]$ is a forest with $\tau_{\max}(G-N[y])=a$. Since $y$ has exactly $r+t$ neighbors and $a\geq 3$, it follows from Theorem~\ref{thm: forest pd} that 
		\[\pd(S/I(G))\leq \max\{a+r+t,\, r+2\}    \leq a+r+t\] which completes the proof.
	\end{proof}
	
	
		\begin{remark}\label{rk: minimality} Let $C$ be a minimal vertex cover of $G$. Then by the minimality, for every $v\in C$, there exists an edge $e$ in $G$ such that $e\cap C=\{v\}$.
	\end{remark}

	\begin{remark}\label{rk: forest max number of edges}
		If $H$ is a forest on $n$ vertices, then it has at most $n-1$ edges.
	\end{remark}
	
	\begin{lemma}\label{lem: trees lemma}
		Let $G$ be a tree on $n$ vertices with $\indm(G)=r$ and $\tau_{\max}(G)=p$. Then $p\leq n-r$.
	\end{lemma}
	
	\begin{proof}
		Assume for a contradiction there exists a minimal vertex cover $C$ of cardinality at least $n-r+1$. Let $M=\{e_1,\dots ,e_r\}$ be an induced matching such that $\cup_{i=1}^ae_i\subseteq C$ and $e_j \not\subseteq C$ for each $j=a+1,\dots ,r$. Since $C$ contains at least $n-r+1$ vertices it follows that $a\geq 1$. Let $U=V(G)\setminus \cup_{i=1}^re_i$. For each $i\in[a]$, let $e_i=\{x_i,y_i\}$. By Remark~\ref{rk: minimality}, for each $i\in [a]$ there exists $u_i,v_i\in U$ such that $\{x_i,u_i\}, \{y_i,v_i\}\in E(G)$ and $u_i,v_i\notin C$.
		
		We claim that $|\cup_{i=1}^a\{u_i,v_i\}|\geq a+1$. Let $H$ be the induced subgraph of $G$ on the vertices $(\cup_{i=1}^ae_i) \cup (\cup_{i=1}^a\{u_i,v_i\})$. Since $H$ is a forest, $3a \leq E(H)\leq V(H)-1$ by Remark~\ref{rk: forest max number of edges}. This proves the claim as $H$ has at least $3a+1$ vertices.
		
		Now, we conclude that $U\cap C$ has at most $n-2r-a-1$ elements. On the other hand, $(V(G)\setminus U)\cap C$ has exactly $r+a$ elements. Thus $C$ has at most $n-r-1$ elements, a contradiction.
	\end{proof}
Let $\tree(n)$ denote the set of all trees on the vertices $\{x_1,\dots ,x_n\}$.  We define
\[\displaystyle \tree_{\pd}^{\reg}(n)=\{(\pd(S/I(G)),\reg(S/I(G))): G \in \tree(n)\}
\]
which consists of all sizes of Betti tables of $S/I(G)$ as $G$ ranges over all trees on $n$ vertices. 

	\begin{theorem}\label{thm: trees (r,p) description}
		Let $G$ be a tree on $n\geq 4$ vertices. Then 
	 \[\tree^{\reg}_{\pd}(n)=\{(p,r)\in \mathbb{Z}^2 : 1\leq r <  n/2, \, \lceil n/2 \rceil \leq p \leq n-r \}.  \]
	\end{theorem}
	\begin{proof}
		 By Lemma~\ref{lem: coarse bound},  Corollary~\ref{cor: even regularity}, and Lemma~\ref{lem: trees lemma} we have
		 \[\tree^{\reg}_{\pd}(n)\subseteq \{(p,r)\in \mathbb{Z}^2 : 1\leq r <  n/2, \, \lceil n/2 \rceil \leq p \leq n-r \}.  \]
		To show the equality, let $1\leq r < n/2$ and $\lceil n/2 \rceil \leq p \leq n-r$ be fixed. By Theorem~\ref{thm: forest pd} and Theorem~\ref{thm: forest reg} it suffices to find a tree $G$ on $n$ vertices with $\indm(G)=r$ and $\tau_{\max}(G)=p$. Since $n> 2r$ we may assume that $n=2r+a$ for some $a\geq 1$. Then since $p\leq n-r$ we obtain $p-a\leq r$. So, we may also assume that $r=p-a+b$ for some $b\geq 0$. Since $p\geq \lceil n/2 \rceil$ we get $p-r\geq 1$. Hence $a-b\geq 1$. Let $t=a-b-1$. Let $G$ be the graph on the vertex set 
		\[V(G)=\{u_1,\dots u_r,v_1,\dots,v_r\} \cup \{x,y_1,\dots,y_b\}\cup\{z_1,\dots ,z_t\}\]
		and the edge set
		\[E(G)=\{\{u_i,v_i\}: i\in [r]\}\cup \{\{x,v_i\}: i\in[r]\}\cup \{\{x,y_i\}:i\in [b]\} \cup \{\{v_r,z_i\}: i\in [t]\}.\] 
		It is clear that $G$ is a tree on $n$ vertices. It is not hard to see that $\{\{u_i,v_i\}:i\in [r]\}$ is an induced matching of maximum cardinality. Moreover, $\{u_1,\dots, u_r, x, z_1,\dots, z_t\}$ is a minimal vertex cover of cardinality $p$. We will now show that $\tau_{\max}(G)=p$. Let $C$ be a minimal vertex cover of $G$. We consider the following cases.
		
		\textit{Case 1}: Suppose that $v_r\notin C$. Then $\{u_r,z_1,\dots ,z_t,x\}\subset C$. Moreover, $y_i\notin C$ for every $i\in [b]$. This implies that for each $i\in [r]$, either $u_i\in C$ or $v_i\in C$, but not both. Hence $|C|=p$.
		
		\textit{Case 2}: Suppose that $v_r \in C$. Then $u_r\notin C$ and $z_i\notin C$ for each $i\in [t]$. We consider two cases:
		
		\textit{Case 2.1}: Suppose that $x\in C$. Then $y_i\notin C$ for each $i\in [b]$. This implies that for each $i\in [r]$, either $u_i\in C$ or $v_i\in C$, but not both. Hence $|C|=r+1\leq p$ as desired.
		
		\textit{Case 2.2}: Suppose that $x\notin C$. Then $C=\{v_1,\dots,v_r, y_1,\dots ,y_b\}$. Thus we get
		\[|C|=r+b=r+(r-p+a)=2r-p+a=2r-p+(n-2r)=n-p \leq p\]
		where the last inequality follows from the assumption that $\lceil n/2 \rceil \leq p$.
\end{proof}
For any positive integer $n$ let $\bpt(n)$ denote the set of connected bipartite graphs on the vertices $\{x_1,\dots ,x_n\}$. We define
\[\displaystyle \bpt_{\pd}^{\reg}(n)=\{(\pd(S/I(G)),\reg(S/I(G))): G \in \bpt(n)\}.
\]	Finally, we arrived at our main result:
\begin{theorem}\label{thm:main theorem}
	Let $n\geq 4$ be an integer. Then
	\[\displaystyle \bpt_{\pd}^{\reg}(n)=\{(p,r)\in \mathbb{Z}^2 : 1\leq r <  \Big\lfloor \frac{n}{2}\Big\rfloor, \, \Big\lceil \frac{n}{2} \Big\rceil \leq p \leq n-2 \}\cup \{(n-1,1)\} \cup A_n \]
	where $A_n=\emptyset$ if $n$ is even and, $A_n=\{(\lceil n/2 \rceil, \lfloor n/2 \rfloor)\}$ if $n$ is odd.
\end{theorem}	
\begin{proof}
	Keeping Lemma~\ref{lem: coarse bound} in mind, first observe that the set $A_n$ is determined by Corollary~\ref{cor: even regularity} and Corollary~\ref{cor: odd regularity}. By Proposition~\ref{prop: pd max value}, $(n-1,1)$ is the only pair in $\bpt_{\pd}^{\reg}(n)$ of the form $(n-1,r)$. The rest of the proof follows from Theorem~\ref{thm: bipartite construction} and Theorem~\ref{thm: trees (r,p) description}.
\end{proof}	
	

\end{document}